\documentclass[12pt,oneside]{amsart}

\usepackage{amssymb,latexsym}

\usepackage{hyperref}
\hypersetup{
  colorlinks   = true, 
  urlcolor     = blue, 
  linkcolor    = blue, 
  citecolor   = red 
}
\usepackage{anysize}
\marginsize{2cm}{2cm}{2cm}{2cm}

\usepackage[pdftex]{graphicx}

\newtheorem{theorem}{Theorem}[section]

\newtheorem{lemma}[theorem]{Lemma}

\newtheorem{conjecture}[theorem]{Conjecture}
\newtheorem{corollary}[theorem]{Corollary}

\theoremstyle{definition}
\newtheorem{definition}[theorem]{Definition}

\theoremstyle{remark}
\newtheorem{remark}[theorem]{Remark}

\numberwithin{equation}{section}

\DeclareMathOperator{\vol}{vol}

\DeclareMathOperator{\cl}{cl}
\DeclareMathOperator{\conv}{conv}

\newcommand{\const}{{\rm const}}

\newcommand{\where}{\mathop{\ |\ }}

\renewcommand{\epsilon}{\varepsilon}
\renewcommand{\phi}{\varphi}
\renewcommand{\kappa}{\varkappa}

\begin{document}

\title{Symplectic polarity and Mahler's conjecture}

\author{Mark Berezovik}

\address{Mark Berezovik, School of Mathematical Sciences, Tel Aviv University, Israel,  69978}

\email{m.berezovik@gmail.com}

\thanks{Mark Berezovik was supported by Horizon Europe ERC Grant number 101045750 (HodgeGeoComb) and partially supported by the ISF grant No. 938/22}

\author{Roman Karasev}

\address{Roman Karasev, Institute for Information Transmission Problems RAS, Bolshoy Karetny per. 19, Moscow, Russia 127994 and Moscow Institute of Physics and Technology, Institutskiy per. 9, Dolgoprudny, Russia 141700}

\email{r\_n\_karasev@mail.ru}
\urladdr{http://www.rkarasev.ru/en/}

\thanks{The research of R. Karasev was carried out within the state assignment of Ministry of Science and Higher Education of the Russian Federation for IITP RAS}


\begin{abstract}
We state a conjecture about the volume of symplectically self-polar convex bodies and show that it is equivalent to Mahler's conjecture concerning the volume of a convex body and its Euclidean polar. We also establish lower and upper bounds for symplectic capacities of symplectically self-polar bodies.
\end{abstract}

\maketitle
\section{Introduction}

This paper is motivated by the (symmetric) Mahler's conjecture from \cite{mahler1939} which asserts that for any centrally symmetric convex body $K\subset \mathbb R^n$ and its polar $K^\circ$ one has
\[
\vol K\cdot \vol K^\circ \ge \frac{4^n}{n!}.
\]
In \cite{aaok2013} it was shown that Mahler's conjecture follows from Viterbo's conjecture \cite{vit2000} in symplectic geometry in the form of the following inequality
\[
\vol X \ge \frac{c_{EHZ}(X)^n}{n!},
\]
where $X\subset\mathbb R^{2n}$ is convex and $c_{EHZ}(X)$ is the Ekeland--Hofer--Zehnder symplectic capacity. In order to confirm Mahler's conjecture it is sufficient to prove Viterbo's conjecture for centrally symmetric convex bodies. Note that Viterbo's conjecture has recently been disproved~\cite{haimkislev2024} for convex bodies that are not centrally symmetric, so the assumption of central symmetry is important here.
 
Our main idea is to state another conjecture of symplectic flavor that turns out to be equivalent (see Theorem~\ref{theorem:mahler-polar}) to Mahler's conjecture. Consider the standard symplectic form $\omega$ in $\mathbb R^{2n}$. Similar to the Euclidean case, the symplectic polar of a convex body $X\subset \mathbb R^{2n}$, containing the origin in its interior, is defined by
\[
X^\omega = \{y\in \mathbb R^{2n} \where \forall x\in X\ \omega(x,y)\le 1\}.
\]
The symplectic polar $X^\omega$ differs from the Euclidean polar $X^\circ$ by a complex rotation $X^\omega = JX^\circ$, where $J$ is the multiplication by $\sqrt{-1}$ under the standard identification $\mathbb C^n\cong\mathbb R^{2n}$. The symplectic polar is a natural and well known object (see, for example, \cite[Section~2.2]{degosson2022} and \cite{degosson2022-2,degosson2023}). Here we mostly study the properties of convex bodies which are symplectically self-polar (boundaries of such bodies in the plane, $2n=2$, are known as \emph{Radon curves} and are widely used, see~\cite{martiniswanepoel2006}).

\begin{conjecture}
\label{conjecture:polar-volume}
For any $X\subset \mathbb R^{2n}$ such that $X=X^\omega$ the following inequality holds
\[
\vol X \ge \frac{2^n}{n!}.
\]
\end{conjecture} 

In the Euclidean case the equation $X=X^\circ$ has a unique solution, the Euclidean unit ball. For the symplectic polarity there are less trivial examples. Consider a convex centrally-symmetric body $K\subset\mathbb R^n$ and its polar $K^\circ$. Take their $\ell_2$-sum $X=K\oplus_2 K^\circ\subset\mathbb R^{2n}$. This is a Lagrangian $\ell_2$-sum, orthogonal in the standard inner product of $\mathbb R^{2n}$. Since the usual polarity takes $\ell_2$-sum to $\ell_2$-sum of polars, then 
\[
X^\circ = K^\circ\oplus_2 K
\] 
and 
\[
X^\omega = JX^\circ = K\oplus_2 K^\circ = X.
\]

These observations lead to the following theorem, whose full proof is given in Section \ref{section:volume-conjecture}.

\begin{theorem}
\label{theorem:mahler-polar}
A lower bound of the form
\[
\vol X \ge c_n \frac{2^n}{n!}
\]
for $X\subset \mathbb R^{2n}$ with $X^\omega = X$ and a sub-exponential $c_n$, for all $n$, is equivalent to the validity of Mahler's conjecture for centrally symmetric bodies in all dimensions. Moreover, if this inequality is valid with a sub-exponential $c_n$, for all $n$, then it is valid in the simple form $\vol X \ge \frac{2^n}{n!}$ as in Conjecture~\ref{conjecture:polar-volume}.
\end{theorem}

\begin{remark}
The upper bound on the volume of a symplectically self-polar body $X\subset\mathbb R^{2n}$, $\vol X\le \frac{\pi^n}{n!}$, follows directly from the Blaschke--Santal\'o inequality \cite{san1949} and is attained by the unit ball.
\end{remark}

\begin{remark}
A less straightforward construction of symplectically self-polar bodies (polytopes) is given in \cite{berezovik2023}. Those polytopes have minimal possible Ekeland--Hofer--Zehnder capacity and conjecturally have minimal possible volume among all symplectically self-polar bodies of given dimension.
\end{remark}

The paper is organized as follows.

\begin{itemize}
\item
In Section~\ref{section:sub-exp-mahler} we show how to suppress sub-exponential factors in Mahler's conjecture. 
\item
In Section~\ref{section:volume-conjecture} we prove Theorem~\ref{theorem:mahler-polar}.
\item
In Section~\ref{section:reduction} we show that Conjecture~\ref{conjecture:polar-volume} may also be approached with the symplectic reduction suggested in \cite{karasev2021} for Mahler's conjecture.
\item
In Section~\ref{section:capacity} we give lower and upper bounds on the symplectic capacity of symplectically self-polar bodies, the lower bound allowing to directly infer a stronger version of Conjecture~\ref{conjecture:polar-volume} from Viterbo's conjecture. 
\item
In Section~\ref{section:affine} we give a non-trivial lower bound on the affine cylindrical capacity of symplectically self-polar bodies and general centrally symmetric convex bodies.
\item
In Appendix~\ref{section:appendix} we present the parts of the unpublished preprint \cite{akopyan2018capacity} that are essentially used in this paper.
\end{itemize}

\subsection*{Acknowledgments} The authors thank Anastasiia Sharipova, Yaron Ostrover, Maurice de Gosson, and the anonymous referees for useful remarks.

\section{Sub-exponential factors in Mahler's inequality}
\label{section:sub-exp-mahler}

In the proof of Theorem~\ref{theorem:mahler-polar} below it turns out that when deducing one conjecture from another there appear sub-exponential factors. In this section we demonstrate how to suppress those factors. This idea seems to be folklore\footnote{See the discussion at \href{https://terrytao.wordpress.com/2008/08/25/tricks-wiki-article-the-tensor-product-trick/}{Terence Tao's blog;} use the search of ``Mahler'' there.}; in relation to Viterbo's conjecture it is used in \cite{haimkislev2021}. For reader's convenience, we state the results that we need explicitly and present their proofs.

\begin{theorem}
\label{theorem:sub-exponential}
If a variant of Mahler's conjecture $($for convex centrally symmetric $K)$ is proved in the form
\[
\vol K\cdot \vol K^\circ \ge c_n \frac{4^n}{n!}
\]
with some sub-exponential $c_n$ $($that is $c_n^{1/n}\to 1$ as $n\to\infty)$ then Mahler's conjecture holds in the original form with $c_n\equiv 1$.
\end{theorem}

Note that the existing lower bounds \cite{bourgainmilman1987,ku2008} are exponentially worse then the conjectured bound. Hence this theorem does not prove Mahler's conjecture, the conjecture remaining open for $n>3$, see \cite{iriyeh2017,fhmrz2019}.

\begin{proof}
Consider a centrally symmetric convex $K\subset\mathbb R^n$ and an integer $m>0$. Consider its Cartesian power $L_m=\underbrace{K\times \dots \times K}_m$ and write 
\[
\vol L_m \cdot \vol L_m^\circ \ge c_{mn} \frac{4^{nm}}{(nm)!},
\]
by the hypothesis of the theorem.

Note that $\vol L_m = (\vol K)^m$ and $L_m^\circ = \underbrace{K^\circ \oplus_1 \dots \oplus_1 K^\circ}_m$ (the $\ell_1$-sum). By a well-known formula stated in Lemma~\ref{lemma:ellp-sum-vol} below, we have
\[
\vol L_m^\circ = \frac{(n!)^m}{(nm)!} \left( \vol K^\circ \right)^m.
\]
Hence we have 
\[
\left( \vol K \cdot \vol K^\circ \right)^m \ge c_{nm} \frac{4^{nm}}{(n!)^m} \Leftrightarrow \vol K \cdot \vol K^\circ \ge c_{nm}^{1/m} \frac{4^n}{n!}.
\]
Going to the limit as $m\to \infty$, we obtain
\[
\vol K\cdot \vol K^\circ \ge \frac{4^n}{n!}.
\]
\end{proof}

\begin{lemma}
\label{lemma:ellp-sum-vol}
If $K\subset\mathbb R^n$ and $L\subset \mathbb R^m$ are convex bodies having the origin in their interior then the $\ell_p$-sum $K\oplus_p L$ has volume
\[
\vol K\oplus_p L = \frac{(n/p)!(m/p)!}{\left(\frac{n+m}{p}\right)!} \vol K\cdot \vol L. 
\]
\end{lemma}
\begin{proof}
Recall the formula for a not necessarily symmetric norm $\|\cdot\|$
\[
\int_{\mathbb R^d} e^{-\|x\|^p}\; dx = \vol X\cdot \left(\frac{d}{p}\right)!,
\] 
where $X\subset \mathbb R^d$ is the unit ball of the norm. This formula follows from considering the volume under the graph of $e^{-\|x\|^p}$ and writing it as an integral over the range $y\in [0,1]$.

Apply the above formula to the norm $(\|x\|_K^p + \|y\|_L^p)^{1/p}$ on $\mathbb R^{n}\times \mathbb R^m$ built from the norms whose unit balls are $K$ and $L$ respectively. The unit ball of this norm is the $\ell_p$ sum $X=K\oplus_p L$. Also, write the integral on the left hand side using Fubini's theorem:
\begin{multline*}
\vol X\cdot \left(\frac{n+m}{p}\right)! = \int_{\mathbb R^{n+m}} e^{-\|(x,y)\|^p}\; dxdy = \\
= \int_{\mathbb R^{n}} e^{-\|x\|_K^p}\; dx\cdot \int_{\mathbb R^{m}} e^{-\|y\|_L^p}\; dy = \vol K\cdot \vol L\cdot \left(\frac{n}{p}\right)!\left(\frac{m}{p}\right)!
\end{multline*}
This is what we need to prove.
\end{proof}

\section{Proof of Theorem \ref{theorem:mahler-polar}}
\label{section:volume-conjecture}

Let us state one simple fact about symplectically self-polar bodies, and then pass to the proof of Theorem \ref{theorem:mahler-polar}.

\begin{lemma}
Every symplectically self-polar body is centrally symmetric.
\end{lemma}
\begin{proof}
The equality $X=X^\omega=JX^\circ$ together with orthogonality of $J$ and the relation $J^2 = -1$ implies
\[
X=(X^\omega)^\omega = J(JX^\circ)^\circ = - (X^\circ)^\circ = - X.
\]
\end{proof}

Assume that Mahler's conjecture holds. Since $X = X^\omega = JX^\circ$, the volumes of $X$ and $X^\circ$ are equal. Applying Mahler's conjecture to $X$ and $X^\circ$, we then obtain with the use of Stirling's formula
\[
\left( \vol X \right)^2 \ge \frac{4^{2n}}{(2n)!}\Rightarrow \vol X \ge c_n \frac{4^n}{\sqrt{(2n)^{2n} e^{-2n}}} = c_n \frac{4^n}{2^n n^n e^{-n}} = c_n \frac{2^n}{n^n e^{-n}} = c'_n \frac{2^n}{n!},
\]
where $c_n$ and $c'_n$ are sub-exponential.

In the other direction, assume that the equality $X=X^\omega$ implies
\[
\vol X \ge c_n \frac{2^n}{n!}
\]
with a sub-exponential $c_n$. Take a convex centrally symmetric $K\subset\mathbb R^n$ and notice that by the well-known Lemma~\ref{lemma:ellp-sum-vol}
\[
\vol K \oplus_2 K^\circ = \vol K\cdot \vol K^\circ \cdot \frac{\left( (n/2)! \right)^2 }{n!}.
\]
Since $X=K\oplus_2 K^\circ$ satisfies $X=X^\omega$, we have 
\[
\vol K\cdot \vol K^\circ \cdot \frac{\left( (n/2)! \right)^2 }{n!} \ge c_n \frac{2^n}{n!}.
\]
Using Stirling's formula we can rewrite
\[
\vol K\cdot \vol K^\circ\ge c_n \frac{n!}{\left( (n/2)! \right)^2 } \frac{2^n}{n!} = c_n \frac{2^n}{\left( (n/2)! \right)^2 } = c'_n \frac{2^n}{(n/2)^n e^{-n}} = c'_n \frac{4^n}{n^n e^{-n}} = c''_n \frac{4^n}{n!}.
\]
This is Mahler's inequality up to a sub-exponential $c''_n$, which implies the precise form of Mahler's inequality by Theorem~\ref{theorem:sub-exponential}.

The last claim of the theorem is proved similarly to the proof of Theorem~\ref{theorem:sub-exponential}. Take $m$ and take $Y = \underbrace{X\oplus_2 \dots \oplus_2 X}_m$. If $X^\omega = X$ then $Y^\omega = Y$, since the $\ell_2$-sum commutes with the polarity and with the $\omega$-duality of centrally symmetric bodies.

By Lemma~\ref{lemma:ellp-sum-vol}, the volumes are related as
\[
\vol Y = \left(\vol X\right)^m \frac{(n!)^m}{(nm)!},
\]
if $n$ is the dimension of $X$. Assuming the weak version
\[
\vol Y \ge c_{nm}\frac{2^{nm}}{(nm)!},
\]
we obtain
\[
\vol X \ge \left( c_{nm} \right)^{1/m} \frac{2^n}{n!}.
\]
Passing to the limit $m\to \infty$, in view of sub-exponentiality of $c_{nm}$, we obtain
\[
\vol X \ge \frac{2^n}{n!}.
\]

\section{Symplectic reduction of self-polar bodies}
\label{section:reduction}

Similar to the results of \cite{karasev2021}, we show that the notion of symplectically self-polar body is compatible with symplectic reduction. By a \emph{linear symplectic reduction of a convex body} we mean a convex body obtained by the following procedure: Take an isotropic linear subspace $V\subset\mathbb R^{2n}$, take its coisotropic $\omega$-orthogonal 
\[
V^{\perp_\omega} = \{y \in \mathbb{R}^{2n}\ |\ \forall x \in V\ \omega(x,y) = 0 \}.
\] 
and let the reduction $Y$ be the projection of the intersection $X\cap V^{\perp_\omega}$ to the quotient $V^{\perp_\omega}/V$ considered as a convex body in the symplectic space $V^{\perp_\omega}/V$.

\begin{lemma}
\label{lemma:self-polar-reduction}
If a convex body $X\subset\mathbb R^{2n}$ satisfies $X=X^\omega$ then any linear symplectic reduction $Y$ of $X$ satisfies $Y=Y^\omega$.
\end{lemma}
\begin{proof}
From the symplectically polar body viewpoint, the section by $V^{\perp_\omega}$ corresponds to the projection along $V$ and the projection along $V$ corresponds to the section by $V^{\perp_\omega}$. Since $V\subseteq V^{\perp_\omega}$, the projection and the section commute, hence the procedure of obtaining the reduction $Y$ from $X$ is the same as the procedure of obtaining $Y^\omega$ from $X^\omega$. Hence $Y=Y^\omega$ whenever $X=X^\omega$.
\end{proof}

\begin{conjecture}
\label{conjecture:self-polar-reduction}
If $X\subset \mathbb R^{2n}$ is symplectically self-polar convex body and $\vol X\ge 2^n/n!$ then any of its linear symplectic reductions $Y\subset\mathbb R^{2n-2}$ has volume $\vol Y\ge 2^{n-1}/(n-1)!$.
\end{conjecture}

\begin{theorem}
\label{theorem:mahker-self-polar-reduction}
Conjecture \ref{conjecture:self-polar-reduction} is equivalent to Mahler's conjecture for centrally symmetric bodies.
\end{theorem}
\begin{proof}
Theorem \ref{theorem:mahler-polar} shows that Mahler's conjecture for centrally symmetric convex bodies is equivalent to the estimate
\[
\vol X \ge \frac{2^n}{n!}
\]
for every symplectically self-polar body $X\subset\mathbb R^{2n}$, hence one direction of the implication is clear. The argument in Section~\ref{section:volume-conjecture} shows in particular that Mahler's conjecture for centrally symmetric convex bodies is also equivalent to the estimate 
\[
\vol K\oplus_2 K^\circ \ge \frac{2^n}{n!}
\]
for all $n$ and all centrally symmetric convex bodies $K\subset\mathbb R^n$. 

Every $K$ in the above inequality can be approximated by $n$-dimensional sections of high-dimensional cubes $Q^N = [-1,1]^N$, hence it is sufficient to prove this estimate for such sections. Denote the polar to the cube, the crosspolytope, by $C^N$. When $K$ is approximated by a section of $Q^N$ with the linear subspace $L\subset\mathbb R^N$ then $K^\circ$ is approximated by the projection of $C^N$ along $L^\perp$. In \cite{karasev2021} it was used that the Lagrangian product $K\times K^\circ$ is then approximated by the linear symplectic reduction of $Q^N\times C^N$. Now we note that $K\oplus_2 K^\circ$ is also approximated by the linear symplectic reduction of $Q^N\oplus_2 C^N$.

For $Q^N$ and $C^N$ in place of $K$ and $K^\circ$ we have
\[
\vol Q^N\oplus_2 C^N = \frac{(N/2)!(N/2)!}{N!}\frac{4^N}{N!} \ge \frac{2^N}{N!},
\]
since $\frac{N!}{(N/2)!(N/2)!} = \binom{N}{N/2} \le 2^N$. Lemma~\ref{lemma:self-polar-reduction} and the assumption that Conjecture~\ref{conjecture:self-polar-reduction} is valid then allow to make a sequence of linear reduction steps and arrive at 
\[
\vol K\oplus_2 K^\circ \ge \frac{2^n}{n!}
\]
in the end.
\end{proof}

\section{Capacities of self-polar bodies}
\label{section:capacity}

We are going to check how Conjecture~\ref{conjecture:polar-volume},
\[
\vol X \ge \frac{2^n}{n!}, \quad X=X^\omega\subset \mathbb R^{2n},
\]
relates to Viterbo's conjectured inequality
\[
\vol X \ge \frac{c_{EHZ}(X)^n}{n!}.
\]
This amounts to checking that $c_{EHZ}(X)\ge 2$ for symplectically self-polar bodies. In fact, we can establish a stronger inequality for a wider class of convex bodies.

\begin{theorem}
\label{theorem:cehz}
Let $X \subset \mathbb{R}^{2n}$ be a centrally symmetric convex body such that $X^{\omega} \subseteq X$. Then
\[
c_{EHZ}(X) \geq 2 + \frac{1}{n}.
\]
\end{theorem}
\begin{proof}
In \cite{akopyan2018capacity} (see Theorem~\ref{theorem:ehz-estimate} in Appendix~\ref{section:appendix}) the following inequality for centrally symmetric convex bodies $X\subset\mathbb R^{2n}$ was established 
\[
c_{EHZ}(X) \ge \left(2 + \frac{1}{n}\right)c_J(X).
\]
The value on the right-hand side is a linear symplectic invariant defined in our current terms as
\[
c_J(X)^{-1} = \max \{|\omega(x,y)|\ |\ x,y\in X^\omega\}.
\]
In \cite{gluskin2015} this invariant was denoted $\|J\|_{K^\circ\to K}^{-1}$, while in \cite{akopyan2018capacity} the above definition and notation $c_J$ was used with $X^\circ$ in place of $X^\omega$. The latter does not matter, since the Euclidean polar and the symplectic polar only differ by the complex rotation $J$, which preserves $\omega$ and does not affect the right hand side.

Evidently, since $X^{\omega} \subseteq X$ then $x\in X^\omega$ and $y\in X$ imply by definition of $X^\omega$ that $|\omega(x,y)|\le 1$ and therefore $c_J(X)\ge 1$. This proves the required estimate.
\end{proof}

In the separate paper \cite[Theorem~1.2]{berezovik2023} it is shown that the estimate of Theorem~\ref{theorem:cehz} is attained on carefully constructed polytopes for any $n$.

Note that in terms of symplectic polarity one may also define $c_J(X)$ for any centrally symmetric convex body $X$ as follows
\[
c_J(X) = \max\{a^{-2}\ |\ (aX)^\omega \subseteq aX\}.
\]
This follows from the observations that both values are $2$-homogeneous in $X$ and $X^\omega\subseteq X$ if and only if $c_J(X)\ge 1$.

Theorem~\ref{theorem:mahler-polar} and Theorem~\ref{theorem:cehz} with the assumption of Viterbo's conjecture hint that the inequality 
\[
\vol X \ge \frac{2^n}{n!}
\]
for $X=X^\omega\subset\mathbb R^{2n}$ need not be tight in any dimension. In particular, we have the following tight bound in dimension 2, attained at the hexagon $\conv\{\pm(0,1), \pm(1,0),\pm(1,1)\}$ and its $\mathrm{SL}(2,\mathbb R)$ images.

\begin{corollary}
\label{corollary:volr2}
If $X\subset \mathbb{R}^2$ is a centrally symmetric convex body and $X^{\omega} \subseteq X$ then $\vol X \geq 3$.
\end{corollary}

Note that Theorem~\ref{theorem:cehz} and Viterbo's conjecture imply the following stronger version of Conjecture~\ref{conjecture:polar-volume}.
\begin{equation}
\label{equation:polar-volume-s}
\vol X \ge \frac{\left(2 + 1/n\right)^n}{n!},\quad \text{when}\quad X^\omega\subseteq X\subset\mathbb R^{2n}.
\end{equation}
This does not contradict Theorem~\ref{theorem:mahler-polar} because the additional multiplier $\left(1+\frac{1}{2n}\right)^n$ is bounded by $e^{1/2}$. So far we only know that this bound \eqref{equation:polar-volume-s} is sharp in the two-dimensional case $n=1$. In higher dimensions one may try to compare this to the one obtained by the symplectically self-polar $\ell_2$-sum $Q^n\oplus_2 C^n$, where $Q^n$ is the cube and $C^n$ is its polar crosspolytope. Its volume can be approximated with Stirling's formula as
\[
\vol Q^n\oplus_2 C^n = \frac{\left(\left(n/2\right)!\right)^2 4^n}{(n!)^2} \sim \frac{n^n e^{-n} 2^n \pi n}{n^n e^{-n} \sqrt{2\pi n} n!} = \sqrt{\frac{\pi n}{2}}\cdot \frac{2^n}{n!}
\]
for $n\to \infty$, which is noticeably larger than \eqref{equation:polar-volume-s}. As one may check by direct calculation, in low dimensions $\vol Q^n\oplus_2 C^n$ is also larger than \eqref{equation:polar-volume-s} by certain amount.

The symplectic reduction of symplectically self-polar bodies allows one to prove the sharp upper bound (attained on balls) on their capacity.

\begin{theorem}
\label{theorem:capacity-bs}
For any centrally symmetric convex body $X\subset \mathbb{R}^{2n}$ one has
\[
c_{EHZ}(X)\cdot c_{EHZ}(X^\omega) \le \pi^2.
\]
In particular, when $X\subset \mathbb{R}^{2n}$ is a symplectically self polar convex body then $c_{EHZ}(X)\le \pi$.
\end{theorem}
\begin{proof}
Let us make a linear symplectic reduction of $X$ to a two-dimensional body $Y$. The argument in the proof of Lemma~\ref{lemma:self-polar-reduction} shows that $Y^\omega$ is the reduction of $X^\omega$. By \cite[Theorem~5.2]{karasev2021} there holds 
\[
c_{EHZ}(X)\le c_{EHZ}(Y) = \vol Y,\quad c_{EHZ}(X^\omega) \le c_{EHZ}(Y^\omega) = \vol Y^\omega.
\]
Now the Blaschke--Santal\'o inequality \cite{san1949} shows that
\[
c_{EHZ}(X)\cdot c_{EHZ}(X^\omega)\le \vol Y\cdot \vol Y^\omega =  \vol Y \cdot \vol JY^\circ = \vol Y \cdot \vol Y^\circ \le \pi^2,
\]
which implies the result.
\end{proof}

\begin{corollary}
\label{corollary:ehzjp}
For any centrally symmetric convex body $X\subset\mathbb R^{2n}$ one has
\[
c_{EHZ}(X) \le \pi c_J(X^\omega)^{-1} = \pi\max\{|\omega(x,y)|\ |\ x,y\in X\}.
\]
\end{corollary}
\begin{proof}
Note that both the left hand side and the right hand side are $2$-homogeneous with respect to scaling of $X$. Then scale $X$ so that $c_J(X^\omega)=1$. By the definition of $c_J$ this means that 
\[
\forall x,y\in X\ |\omega(x,y)|\le 1
\]
and the equality sometimes attained. This may be 
$X\subseteq X^\omega$. From the monotonicity of the capacity and Theorem~\ref{theorem:capacity-bs} we have
\[
c_{EHZ}(X)^2 \le c_{EHZ}(X)\cdot c_{EHZ}(X^\omega) \le \pi^2\Rightarrow c_{EHZ}(X)\le \pi,
\]
which finishes the proof.
\end{proof}

\section{Lower bounds on the affine cylindrical capacity}
\label{section:affine}

Let us draw some higher-dimensional consequences from Corollary~\ref{corollary:volr2}. Consider a symplectic linear subspace $L \subset \mathbb{R}^{2n}$ and the $\omega$-orthogonal decomposition $\mathbb{R}^{2n} = L \oplus L^{\perp \omega}$. Let the projection onto $L$ along $L^{\perp \omega}$ be $\pi_{L,\omega}$, and the restriction of $\omega$ to $L$ be $\omega_L$ for brevity.

\begin{lemma}
\label{lemma:proj}
Let $X \subset \mathbb{R}^{2n}$ be a centrally symmetric convex body such that $X^{\omega} \subseteq X$. Then $Y = \pi_{L,\omega}(X) \subset L$ in the above notation is centrally symmetric and $Y^{\omega_L} \subseteq Y$. 
\end{lemma}
\begin{proof}
It is easy to see that $Y^{\omega_L} \subseteq \pi_{L,\omega}(X^{\omega})$ and therefore 
\[
Y^{\omega_L} \subseteq \pi_{L,\omega}(X^{\omega}) \subseteq \pi_{L,\omega}(X) = Y.
\]
\end{proof}

Consider a variant of symplectic capacity which is only affine-invariant (not invariant under non-linear symplectic diffeomorphism as demanded in the abstract axioms of a symplectic capacity). This was called $\bar c_{lin}$ and related to the Ekeland--Hofer--Zehnder and other capacities in \cite{gluskin2015}. Here we use a different notation with larger letters.

\begin{definition}
The \emph{cylindrical affine capacity} of a convex body $X \subset \mathbb{R}^{2n}$ is the value
\[
c_{ZA}(X) = \inf \vol(\pi_{L,\omega}(X)),
\]
where the infimum is taken over symplectic two-dimensional linear subspaces $L \subset \mathbb{R}^{2n}$ and the volume is defined with $\omega_L$.
\end{definition}

Now we show, in particular, that $c_{ZA}$ for symplectically self-polar bodies has a better lower bound than $c_{EHZ}$.

\begin{theorem}
\label{theorem:za3}
Let $X \subset \mathbb{R}^{2n}$ be a centrally symmetric convex body such that $X^{\omega} \subseteq X$, then
\[
c_{ZA}(X) \geq 3.
\]
\end{theorem}
\begin{proof}
Take a two-dimensional linear symplectic subspace $L \subset \mathbb{R}^{2n}$. By Lemma~\ref{lemma:proj} the projection $Y = \pi_{L,\omega}(X)$ satisfies $Y^{\omega_L} \subseteq Y$. Then Corollary~\ref{corollary:volr2} implies that $\vol Y \ge 3$.
\end{proof}

This result may be restated without a mention of the symplectic polar, close to the statement of \cite[Theorem 1.1]{akopyan2018capacity}.

\begin{corollary}
\label{corollary:zaj}
For any centrally symmetric convex body $X\subset\mathbb R^{2n}$ one has
\[
c_{ZA}(X)\ge 3 c_J(X).
\]
\end{corollary}
\begin{proof}
Note that both the left hand side and the right hand side are $2$-homogeneous with respect to scaling of $X$. Then scale $X$ so that $c_J(X)=1$. By the definition of $c_J$ this means that $X^\omega \subseteq (X^\omega)^\omega = X$, and therefore Theorem~\ref{theorem:za3} finishes the proof.
\end{proof}

We also have a corollary showing that if one wants to use $c_{ZA}$ in place of $c_{EHZ}$ in Viterbo's conjecture then the statement should be weakened. 

\begin{corollary}
\label{corollary:affine-anti-viterbo}
If the inequality
\[
\vol(X) \geq b^n \frac{c_{ZA}(X)^n}{n!}
\]
holds for any $n$, any centrally symmetric convex $X \subset \mathbb{R}^{2n}$, and a constant $b$ then $b \leq 2/3$.
\end{corollary}
\begin{proof}
Take the cube $Q^n = [-1,1]^n$, the cross-polytope $C^n = (Q^n)^\circ$, and set $X = Q^n\oplus_2 C^n$. This is a symplectically self-polar body with $c_{ZA}(X)\ge 3$ by Theorem~\ref{theorem:za3}. Its volume by Lemma~\ref{lemma:ellp-sum-vol} and Stirling's formula is
\[
\vol Q^n \oplus_2 C^n = \frac{4^n}{n!} \cdot \frac{\left( (n/2)! \right)^2 }{n!} = c_n \frac{4^n n^2 2^{-n} e^{-n}}{n! n^n e^{-n}} = c_n \frac{2^n}{n!}
\]
up to sub-exponential $c_n$. This shows that the exponential factor $b^n$ cannot be greater than $(2/3)^n$.
\end{proof}

Let us show the sharpness of Theorem~\ref{theorem:za3} and Corollary~\ref{corollary:zaj} in the following strong sense.

\begin{theorem}
There exist symplectically self-polar bodies $X\subset\mathbb R^{2n}$ in arbitrary dimension such that $c_{ZA}(X)=3$.
\end{theorem}
\begin{proof}
Take the planar hexagon $P=\conv\{\pm(0,1), \pm(1,0),\pm(1,1)\}$ and put it into the standard symplectic $2$-plane $\mathbb R^2\subset \mathbb R^{2n}$. Note that 
\[
P^\omega = P\times \mathbb R^{2n-2}. 
\]
These are unbounded convex bodies, but the definition and the properties of the polar apply. We may produce from them a convex body $Q = \conv (P\cup B_0(\varepsilon))$ and its symplectic polar $Q^\omega = P^\omega\cap B_0(1/\varepsilon)$. For sufficiently small $\varepsilon>0$, the projection of $Q$ to the linear span of $P$ equals $P$.

By Lemma~\ref{lemma:inside} (see below) there exists a symplectically self-polar body $X$ such that 
\[
Q\subseteq X\subseteq Q^\omega \subset P\times \mathbb R^{2n-2}.
\]
Obviously, the projection of $X$ to the linear span of $P$ is $P$ and therefore $c_{ZA}(X)\le 3$.
\end{proof}

\begin{lemma}
\label{lemma:inside}
If $Y \subset \mathbb{R}^{2n}$ is a centrally symmetric convex body such that $Y^{\omega} \subseteq Y$ then there exists a symplectically self-polar body $X \subset \mathbb{R}^{2n}$ such that $Y^\omega\subseteq X^\omega = X \subseteq Y$.
\end{lemma}

\begin{proof}
Let the inclusion $Y^{\omega} \subset Y$ be strict. Take a point $p \in Y \setminus Y^{\omega}$ and consider the convex body 
\[
Z = Y \cap \{x \in \mathbb{R}^{2n}\ |\ |\omega(x,p)| \leq 1\}, 
\]
then 
\[
Z^{\omega} = \conv(Y^{\omega}\cup\{\pm p\}) \subseteq Y. 
\]
Any point $z \in Z^{\omega}$ can be represented as $z = ty \pm (1-t)p$, where $y \in Y^{\omega}, t \in [0,1]$. Since $|\omega(y,p)| \leq 1$ and $\omega(p,p) = 0$, it follows that $|\omega(z,p)| \leq 1$. Hence 
\[
Y^{\omega} \subset Z^{\omega} \subseteq Z \subset Y.
\]
	
Now consider the family of all convex bodies $Z \subseteq Y$ such that $Z^{\omega} \subseteq Z$ ordered by inclusion. Let us show that it satisfies the hypothesis of the Zorn lemma. Let $\{Z_{\alpha}\}_{\alpha \in I}$ be a chain of such sets. Set $Z = \bigcap_{\alpha \in I} Z_{\alpha}$, then
\[
Z^{\omega} =  JZ^{\circ} = \left(\bigcap_{\alpha \in I} JZ_{\alpha}\right)^{\circ} = \cl\bigcup_{\alpha \in I} JZ_{\alpha}^{\circ} = \cl\bigcup_{\alpha \in I} Z_{\alpha}^{\omega}. 
\]
Let $Z_\alpha\subseteq Z_\beta$ be two members of that chain, then $Z_\alpha^\omega\supseteq Z_\beta^\omega$. Fixing $\beta$ and using the chain property one sees that
\[
\bigcup_{\alpha \in I} Z_{\alpha}^{\omega} = \bigcup_{\alpha \in I,\ Z_\alpha\subseteq Z_\beta} Z_{\alpha}^{\omega}.
\]
Since $Z_{\alpha}^{\omega} \subseteq Z_\alpha \subseteq Z_\beta$ whenever $Z_\alpha\subseteq Z_\beta$, one has 
\[
\bigcup_{\alpha \in I} Z_{\alpha}^{\omega} = \bigcup_{\alpha \in I,\ Z_\alpha\subseteq Z_\beta} Z_{\alpha}^{\omega}\subseteq Z_\beta.
\]
Since this holds for any $\beta\in I$, we may pass to the intersection and write
\[
\bigcup_{\alpha \in I} Z_{\alpha}^{\omega}\subseteq  Z.
\]	
Since $Z$ is closed, then the inclusion is preserved after taking the closure, and finally 
\[
Z^{\omega} = \cl\bigcup_{\alpha \in I} Z_{\alpha}^{\omega} \subseteq Z.
\]
By the Zorn lemma the family contains a minimal element $X$. The argument in the beginning of the proof (for $X$ in place of $Y$) shows that $X^{\omega} = X$ must hold assuming the minimality.
\end{proof}

\section{Appendix}
\label{section:appendix}

Here we present the part of \cite{akopyan2018capacity} that was used above to make this exposition self-contained. 

\begin{theorem}[The symmetric case of Theorem 1.1 in \cite{akopyan2018capacity}]
\label{theorem:ehz-estimate}
For a centrally symmetric convex body $X\subset \mathbb R^{2n}$
\[
c_{EHZ}(X) \ge \left(2 + \frac{1}{n}\right) c_J(X).
\]
\end{theorem}
\begin{proof}
We may assume that $X$ is smooth since both sides of the inequality are continuous in $X$ and the non-smooth case follows by approximation. Following~\cite{gluskin2015}, we consider a closed characteristic $\gamma : [0, T] \to \partial X$ satisfying the equation
\begin{equation}
\label{equation:characteristic}
\dot\gamma = J\nabla g_X(\gamma)
\end{equation}
and having the minimal possible action $A(\gamma) = c_{EHZ}(X)$. Here $g_X$ is the gauge function of $X$ (so that $X=\{z\in\mathbb R^{2n}\ |\ g_X(z)\le 1\}$), $J$ is the complex rotation. Note that since $g_X$ is $1$-homogeneous then the action is
\begin{multline*}
c_{EHZ}(X) = A(\gamma)=\frac{1}{2} \int_0^T \omega(\gamma(t),\dot\gamma(t)) \; dt = - \frac{1}{2} \int_0^T \gamma(t)\cdot J \dot \gamma(t) \; dt = \\
= \frac{1}{2} \int_0^T \gamma(t)\cdot \nabla g_X(\gamma(t)) \; dt = \frac{1}{2} \int_0^T g_X(\gamma(t)) \; dt = T/2. 
\end{multline*}
For centrally symmetric $X$, at least one of the optimal characteristics $\gamma$ (we choose this one) has to be centrally symmetric with respect to the origin by Lemma~\ref{lemma:symmetry}. Then Lemma~\ref{lemma:schaffer} (see below) yields the estimate
\[
\int_0^T g_X(\dot\gamma)\; dt \ge 4 + \frac{2}{n}.
\]
We conclude the proof as in~\cite{gluskin2015},
\[
4 + \frac{2}{n} \le \int_0^T g_X(\dot\gamma)\; dt \le \int_0^T g_X(J\nabla g_X(\gamma))\; dt \le \int_0^T c_J(X)^{-1}\; dt = \frac{T}{c_J(X)} = \frac{2 A(\gamma)}{c_J(X)}.
\]
Here we use that $y = \nabla g_X\in X^\circ$ from $1$-homogeneity of $g_X$ and 
\[
g_X(Jy) = \max\{ Jy\cdot x\ |\  x\in X^\circ\} = \max\{ \omega(y, x)\ |\  x\in X^\circ\} \le c_J(X)^{-1}
\]
by the definition of $c_J(X)$.
\end{proof}

In order to establish one of the needed lemmas, we are going to use the fact, proved in \cite{clarke1979}, that the closed characteristics of smooth $\partial X$ are affine images of the solutions of the following variational problem for closed curves $\gamma : \mathbb R/\mathbb Z \to \mathbb R^{2n}$
\begin{equation}
\label{equation:clarke}
\int_\gamma h_X^\omega(\dot\gamma) \to \min, \quad \int_\gamma \lambda = 1,
\end{equation}
where $\lambda$ is a primitive of $\omega$ and 
\[
h_X^\omega(y) = \max\{\omega(x,y)\ |\ x\in X\}
\]
is the symplectic support function of $X$.

\begin{lemma}[Lemma 2.1 in \cite{akopyan2018capacity}]
\label{lemma:symmetry}
In the variational problem \eqref{equation:clarke}, for centrally symmetric convex body $X$, one of the minima is attained at a curve $\gamma$ centrally symmetric with respect to the origin. If $X$ is smooth then the corresponding characteristic of $X$ is also centrally symmetric with respect to the origin.
\end{lemma}
\begin{proof}
Assume $\gamma$ is the minimal curve, split it into two curves $\gamma_1$ and $\gamma_2$ of equal $h_X^\omega$-lengths. We are going to use the particular primitive $\lambda = \sum_i p_idq_i$, which is invariant under the central symmetry of $\mathbb R^{2n}$.
	
Since the problem is invariant under translations of $\gamma$, we may translate it so that $\gamma_1$ passes from $-x$ to $x$ and $\gamma_2$ passes from $x$ to $-x$. Let $\sigma$ be the straight line segment from $x$ to $-x$ and $-\sigma$ be its opposite. Then the concatenations $\beta_1 = \gamma_1\cup\sigma$ and $\beta_2 = \gamma_2\cup (-\sigma)$ are the closed loops such that
\[
\int_{\beta_1} \lambda + \int_{\beta_2} \lambda = \int_{\gamma} \lambda = 1.
\]
Then without loss of generality we assume $\int_{\beta_1} \lambda \ge 1/2$. Then the centrally symmetric curve $\gamma' = \gamma_1\cup(-\gamma_1)$ has
\[
\int_{\gamma'} \lambda = \int_{\beta_1} \lambda + \int_{-\beta_1} \lambda \ge 1,
\]
and the $h_X^\omega$-length of $\gamma'$ is the same as the length of $\gamma$. Scaling $\gamma'$ to obtain $\gamma''$ with $\int_{\gamma''} \lambda = 1$ we will have centrally symmetric $\gamma''$ with length no greater than the length of the original $\gamma$.

We have established that $\gamma$ is centrally symmetric with respect to the origin and need to show the same for its corresponding characteristic of $X$. Note that in \cite{clarke1979} the variational problem used a $2$-homogeneous Lagrange function $L(\dot\gamma) = 1/2 (h_X^\omega(\dot\gamma))^2$ instead of $1$-homogeneous function $h_X^\omega(\dot\gamma)$ in our version \eqref{equation:clarke}. But those two versions are standardly equivalent since passing from arbitrary curve $\gamma$ to its reparametrization such that $h_X^\omega(\dot\gamma)=\const$ may only decrease the integral of $2$-homogeneous $L$ because of the Cauchy--Schwarz inequality and keeps the integral of $h_X^\omega(\dot\gamma)$ the same. In particular, the solutions to both problems may be assumed parametrized so that $h_X^\omega(\dot\gamma)=\const$.

The Euler--Lagrange equation from \cite{clarke1979} can be written as (after changing $p\mapsto -p$ to suppress the extra minus in the notation of \cite{clarke1979})
\[
\nabla L(\dot\gamma) + \lambda J \gamma(t) = c
\]
for constant $c\in\mathbb R^{2n}$ and $\lambda\neq0\in\mathbb R$. For a centrally symmetric $\gamma$ with period $T$, we have $\gamma(t+T/2)=-\gamma(t)$ and $\dot\gamma(t+T/2)=-\dot\gamma(t)$ for all $t$. Since $L$ is even and its gradient is odd, we see that the left hand side flips the sign when passing from parameter $t$ to $t+T/2$, while the right hand side is a constant. Hence we have $c=0$. Since $c$ is the constant term in the affine transformation that produces a characteristic of $X$ from $\gamma$ in \cite{clarke1979} then the affine transformation is in fact linear and the characteristic is also centrally symmetric with respect to the origin.
\end{proof}

The following lemma about curves on the unit spheres of normed spaces is essentially Theorems 13E and 13F in the book \cite{schaffer1976}.

\begin{lemma}[Schaffer, 1976]
\label{lemma:schaffer}
If $X\subset\mathbb R^d$ is a centrally symmetric convex body and $\gamma\subset \partial X$ is a closed curve centrally symmetric with respect to the origin then the length of $\gamma$ in the norm whose unit ball is $X$ is at least $4 + 4/d$. For odd $d$ the bound can be improved to $4+4/(d-1)$.
\end{lemma}

\bibliography{../Bib/karasev}
\bibliographystyle{abbrv}
\end{document}